\documentclass{article}
\usepackage{maa-monthly}

\theoremstyle{theorem}
\newtheorem{theorem}{Theorem}

\newtheorem{corollary}{Corollary}
\newtheorem{lemma}{Lemma}

\theoremstyle{definition}
\newtheorem*{definition}{Definition}

\begin{document}

\title{Infinitely Many Twin Prime Polynomials of Odd Degree}
\markright{Twin Prime Polynomials}
\author{Claire Burrin and Matthew Issac}

\maketitle

\begin{abstract}
While the twin prime conjecture is still famously open, it holds true in the setting of finite fields: There are infinitely many pairs of monic irreducible polynomials over $\mathbb{F}_q$ that differ by a fixed constant, for each $q\geq3$. Elementary, constructive proofs were given for different cases by  Hall and Pollack. In the same spirit, we discuss the construction of a further infinite family of twin prime tuples of odd degree, and its relations to the existence of certain Wieferich primes and to arithmetic properties of the combinatorial Bell numbers.
\end{abstract}

\section{Introduction.}

Let $\mathbb{F}_q$ be a finite field of $q\geq3$ elements, where $q$ is a prime power. The ring of integers $\mathbb{Z}$ and the polynomial ring $\mathbb{F}_q[X]$ exhibit a number of common features, including both being unique factorization domains. A prime (polynomial) in the latter setting is a monic irreducible polynomial. Our understanding of the distribution of prime polynomials is significantly more complete. To start with, one can precisely count the number $\pi_q(n)$ of monic irreducible polynomials of degree $n$ over $\mathbb{F}_q$, as was done by Gauss, who proved that 
$$
\pi_q(n) = \frac{1}{n}\sum_{d|n}\mu(d)q^{\tfrac{n}{d}},
$$ 
where $\mu(d)$ is the classical M\"obius function. As a result, as $q^n\to\infty$,
$$
\pi_q(n) = \frac{q^n}{n} + O\left(\frac{q^{n/2}}{n}\right),
$$
and this should be contrasted with the classical problem of counting prime numbers, for which the Riemann hypothesis is equivalent to the assertion that the number $\pi(x)$ of prime numbers less or equal than $x$ is
$$
\pi(x) = \frac{x}{\log x} + O\left(x^{1/2+o(1)}\right),
$$
as $x\to\infty$.

Another famous, long-standing open problem of number theory with a happier resolution over finite fields is the twin prime conjecture. For integers, the conjecture is that there are infinitely many pairs of primes of the form $(p,p+2)$ and this is still open, despite the spectacular breakthroughs of Zhang \cite{Zh} and Maynard \cite{Ma}. A refined quantitative form of this conjecture due to Hardy and Littlewood asserts that given distinct integers $a_1,\dots,a_r$, the number $\pi(x;a_1,\dots,a_r)$ of integers $n\leq x$ for which $n+a_1,\dots,n+a_r$ are simultaneously prime is 
$$
\pi(x;a_1,\dots,a_r)\sim \frak{S}(a_1,\dots,a_r)\frac{x}{(\log x)^r}
$$
as $x\to\infty$, for a nonnegative constant $\frak{S}(a_1,\dots,a_r)$ encoding local congruence obstructions. 

To formulate the corresponding problems over finite fields, let $q\geq3$. The \emph{size} of a nonzero polynomial $f$ of degree $n$ over $\mathbb{F}_q$ is defined to be $$|f|_q := q^n=|\mathbb{F}_q[X]/(f)|.$$ Two prime polynomials $f,g\in\mathbb{F}_q[X]$ form a \emph{twin prime pair} if the size of their difference $|f-g|_q$ is as small as possible, namely, if $|f-g|_q=1$. 
The twin prime conjecture asks for the existence of infinitely many twin prime pairs $(f,f+a)$ for some fixed $a\in\mathbb{F}_q^\times$. Given positive integers $n$ and $r$, and given distinct polynomials $a_1,\dots,a_r\in\mathbb{F}_q[X]$, each of degree less than $n$,
the Hardy--Littlewood conjecture asks for the growth of the number $\pi_q(n;a_1,\dots,a_r)$ of  tuples $(f+a_1,\dots,f+a_r)$, with $|f|_q=q^n$, as $q^n\to\infty$. There are two ways in which $q^n\to\infty$, namely taking either $q\to\infty$ or $n\to\infty$. These limits are usually considered separately, as they may lead to different asymptotic behaviors. Bary-Soroker \cite{BS2} proved that for any fixed positive integers $n$ and $r$, and distinct $a_1,\dots,a_r\in\mathbb{F}_q[X]$, each of degree less than $n$,
\begin{align}\label{HL}
\pi_q(n;a_1,\dots,a_r) = \frac{q^n}{n^r}+O_{n,r}(q^{n-1/2})
\end{align}
as $q\to\infty$, and very recently, Sawin and Shusterman \cite{SawinShusterman} settled the Hardy--Littlewood conjecture over finite fields by showing that for every large enough odd prime power $q$,
$$
|\{f\in\mathbb{F}_q[X]: |f|_q=x,\ (f,f+a) \text{ twin prime pair}\}|  \sim \frak{S}_q(a)\frac{x}{(\log_q x)^2}
$$
as $x\to\infty$ through powers of $q$, and where $\frak{S}_q(a)$ is the function field analogue of $\mathfrak{S}(a)$. 
%In this note, we present some {\em elementary} constructions of infinite families of twin prime pairs (or tuples) of polynomials.

Interestingly, it is also possible to showcase infinite families of twin prime pairs (or tuples) of polynomials. In this article, we consider \emph{elementary} constructions of this sort.

\section{An elementary proof.}
In his Ph.D. thesis \cite{Hall1}, Hall observed that over most finite fields, the twin prime conjecture in its qualitative form is an easy consequence of the following classical result of field theory; see, e.g., \cite[Theorem 9.1, p.~297]{Lang}.
\begin{theorem}\label{Lang}
Let $F$ be a field. Fix $n\in\mathbb{N}$ and $a\in F^\times$.  Then $X^n-a\in F[x]$ is irreducible over $F$ if and only if 
\begin{enumerate}
\item[(a)]
$a\not\in F^\ell=\{ a^\ell : a\in F\}$ for each prime divisor $\ell|n$, 
\item[(b)] and $a\not\in -4F^4$ whenever $4|n$.
\end{enumerate}
\end{theorem}

\begin{corollary}\cite[Corollary 19]{Hall1}\label{Hall}
If $q-1$ admits an odd prime divisor, then there are infinitely many twin prime pairs $(f,f+1)$ over $\mathbb{F}_q$.  
\end{corollary}

\begin{proof} 
Let $\ell\mid q-1$ be an odd prime, and let $(\mathbb{F}_q^\times)^\ell$ be the subgroup of $\ell$th powers of elements in the unit group $\mathbb{F}_q^\times$. Since
$
|(\mathbb{F}_q^\times)^\ell|=\frac{q-1}{\ell} <\tfrac{q-1}{2},
$
there exist two consecutive elements $a, a-1\not\in (\mathbb{F}_q^\times)^\ell$ by Dirichlet's pigeonhole principle. Then by Theorem \ref{Lang}, $(X^{\ell^m}-a,X^{\ell^m}-a+1)$ is a twin prime pair, and this for each $m\geq0$.
\end{proof}

This remarkably simple proof settles the twin prime conjecture for all finite fields $\mathbb{F}_q$ save for the cases where
\begin{align}\label{generic}
q=2^n+1,
\end{align}
for some $n\in\mathbb{N}$. We observe that for prime fields of this form, $q$ is a Fermat prime. To this day, the only known Fermat primes are 3, 5, 17, 257, 65,537 and conjecturally, only finitely many exist. On the other hand, Catalan's conjecture (proved by Mih\u{a}ilescu \cite{Mi}) asserts that $2^3$ and $3^2$ are the only two existing consecutive positive powers, and hence $q=9$ is the only admissible prime power of the form (\ref{generic}).

\begin{definition}
A finite field $\mathbb{F}_q$ is called \emph{generic} if the order $|\mathbb{F}_q^\times|=q-1$ of the unit group $\mathbb{F}_q^\times$ has at least one odd prime divisor. Otherwise, it is called \emph{nongeneric}.
\end{definition}

For nongeneric finite fields (and in fact, more generally when $q\equiv 1$ (mod 4)), part (b) of Theorem \ref{Lang}, together with elementary counting considerations for quadratic residues, yields an explicit infinite family of twin prime polynomials of even degree; see \cite{Pol}. This construction has been extended to obtain twin prime tuples; see \cite{Eff}.

\section{A refined problem.}
Every student who took an introductory number theory class will be familiar with the following question: Given that there are infinitely many primes, and that each odd prime is congruent to either 1 or 3 (mod 4), are there infinitely many primes $p$ such that $p\equiv 1$ (mod 4), or, respectively, such that $p\equiv 3$ (mod 4)? Similarly, one may ask whether there exist infinitely many twin prime pairs $(f,f+1)$ of odd (respectively, even) degree?  

The question was settled affirmatively in the Ph.D. thesis of Pollack \cite{Pol}. For large enough $q$, the asymptotic (\ref{HL}) implies a positive answer, and Pollack's strategy was to bootstrap such an asymptotic to a substitution procedure, and treat the cases where $q$ is small by hand. The case of even degree can actually be approached directly, relying on a number of elementary constructions; see \cite[Lemmas 6.3.2--4]{Pol}. In the rest of this  note, we wish to consider a new elementary construction, in the spirit of Corollary \ref{Hall}, to cover the odd degree case.

Clearly, this is already achieved for generic fields $\mathbb{F}_q$ by the proof of Corollary \ref{Hall}. To cover also nongeneric fields, we examine below a different construction over prime fields $\mathbb{F}_p$. This does not leave out the case $\mathbb{F}_9$, as we explain next. In fact, any infinite family of twin primes of odd degree over $\mathbb{F}_3$ also defines an infinite family of twin primes of odd degree over $\mathbb{F}_9$. This follows from the following standard result for polynomials over finite fields: An irreducible polynomial of degree $n$ over $\mathbb{F}_q$ is irreducible over $\mathbb{F}_{q^k}$ if and only if $(k,n)=1$; see, e.g., \cite[Corollary 3.47]{LN}. Since in our situation, $k=2$ and $n$ is an odd degree, the odd twin prime conjecture for $\mathbb{F}_9[X]$ reduces to the case of $\mathbb{F}_3[X]$.

Let $p$ be an odd prime. The starting point of our construction is the polynomial 
$$
f(X)=X^p-X-1.
$$ 
By the Artin--Schreier theorem, $f(X)$ is irreducible over $\mathbb{F}_p$. Thus if $\alpha$ is a root, then $\mathbb{F}_p(\alpha)$ is a cyclic Galois extension of degree $p$ over $\mathbb{F}_p$; see \cite[Theorem 6.4, p. 290]{Lang}. Hence $\mathbb{F}_p(\alpha)\cong\mathbb{F}_{p^p}$. In particular, all roots $\alpha, \alpha^p,\dots,\alpha^{p^{p-1}}$ of $f$ are Galois conjugates and have the same multiplicative order in $\mathbb{F}_{p^p}^\times$. The order $e$ of the polynomial $f(X)$ is defined to be the multiplicative order of any of its roots in $\mathbb{F}_{p^p}^\times$. To examine this order $e$, we observe that 
$$
f(0) = -1 =\prod_{i=0}^{p-1} (0-\alpha^{p^i})=(-\alpha)^{1+p+\dots+p^{p-1}}=-\alpha^Q,
$$
where
$$
Q:= 1+p+p^2+\dots+p^{p-1} = \frac{p^p-1}{p-1}.
$$
It easily follows that $e \mid Q$. We have $(Q, 2p(p-1)) = 1$ directly from the definition of $Q$, and hence $e$ is odd and relatively prime to $p(p-1)$. 
Less obviously, the order $e$ coincides with the minimal period of Bell numbers modulo $p$. The Bell number $B(n)$ is the number of distinct partitions of a finite set of $n$ elements. A great number of problems can be interpreted in terms of Bell numbers; among other things, $B(n)$ counts
\begin{itemize}
\item the number of equivalence relations among $n$ elements,
\item the number of factorizations of the product of $n$ distinct primes into coprime factors,
\item the number of permutations of $n$ elements with ordered cycles;
\end{itemize}
see \cite{Rota} and references therein. Determining this minimal period has attracted quite a bit of attention. For very small primes ($p<180$), numerical computations show that $e=Q$, with some further probabilistic evidence given in \cite{Mo}.

To state our main result, we recall that a prime $\ell$ satisfying the congruence equation $b^{\ell-1}\equiv1$ (mod $\ell^2$), where $(b,\ell)=1$, is called a \emph{Wieferich prime in base $b$}.

\begin{theorem}\label{thm}
Let $p$ be an odd prime. For each $a\in\mathbb{F}_p^\times$, set $f_a(X)=X^p-X+a$ and let $e$ denote the order of $f_{-1}(X)$. For each odd prime divisor $\ell\mid e$,
\begin{enumerate}
\item 
if $\ell\nmid\tfrac{p^p-1}{e}$, then 
\begin{align}\label{family}
\{(f_1(X^{\ell^m}),f_2(X^{\ell^m}),\dots,f_{p-1}(X^{\ell^m})): m\geq0\}
\end{align}
is an infinite family of twin prime tuples of odd degree over $\mathbb{F}_p$;
\item if $\ell\mid \tfrac{p^p-1}{e}$, then $\ell$ is a Wieferich prime in base $p$.
\end{enumerate}
\end{theorem}

We note that the conjecture $e=Q$ would in particular imply that for each $\ell\mid e$, (\ref{family}) forms an infinite family of twin prime tuples over $\mathbb{F}_p$. The proof of Theorem \ref{thm} is elementary; we postpone it to Section \ref{sec3} and discuss here the problem of the existence of Wieferich primes. 

The fame of Wieferich primes in number theory owes to their appearance in work on Fermat's last theorem. In 1909, Wieferich \cite{Wi} proved that if the first case of Fermat's last theorem is false, i.e., if $X^p+Y^p=Z^p$ is solvable in positive integers $X$, $Y$, $Z$ for an odd prime $p$ such that $(p,XYZ)=1$, then $p$ must be a Wieferich prime in base 2. A year later, Miramanoff reached the same conclusion for base 3. An arms race was engaged to prove that up to large $x$, no prime below $x$ is simultaneously a Wieferich prime in base 2 and in base 3. In fact, numerically, Wieferich primes are rare: in base 2, the only ones presently known \cite{DK} below $6.7\times 10^{15}$ are 1093 and 3511, while in base 47, there is simply no known Wieferich prime. Heuristically, if we consider $\tfrac{b^\ell -1}{\ell}$ as a random integer, the probability that $\ell\mid \tfrac{b^\ell -1}{\ell}$ is roughly $1/\ell$. Since
 $$
 \sum_{\ell\leq x} \frac{1}{\ell} \ll \log\log x,
 $$
 this heuristic suggests that the number of Wieferich primes up to $x$ in base $b$ is of the order of the iterated logarithm $\log\log x$. The iterated logarithm tends to $\infty$ as $x\to\infty$, but it does so very, very slowly; e.g., if $x=10^{100}$, then $\log\log x\approx 5.4$. For comparison, the number of atoms in the universe is roughly of the order of $10^{80}$. As such, we expect that for every base $b$, there are infinitely many Wieferich primes as well as infinitely many non-Wieferich primes. (The latter, see \cite{Sil}, is not even known unless one assumes the $abc$-conjecture.) 

Fermat's last theorem is not the only place where Wieferich primes appear as obstructions. Fermat and Mersenne numbers, i.e., $F_n=2^{2^n}+1$ and $M_n=2^n-1$, $n\in\mathbb{N}$, are believed to be squarefree. In trying to prove this directly, one quickly sees that any prime factor $p$ such that $p^2$ divides either $F_n$ or $M_n$ must be a Wieferich prime in base 2. Theorem \ref{thm} showcases a similar phenomenon.

\section{Proof of Theorem \ref{thm}.}\label{sec3}
We quickly recall elements of notation. Let $p$ be an odd prime. Let $f_a(X)=X^p-X+a$, for $a\in\mathbb{F}_p^\times$. The order $e$ of $f(X):=f_{-1}(X)$ %is the multiplicative order of any of its root in $\mathbb{F}_{p^p}$. It 
satisfies $e\mid Q=\tfrac{p^p-1}{p-1}$ and $(e,2p(p-1))=1$. Fix $\ell\mid e$ prime, and note that $\ell$ is necessarily odd.

Suppose first that $\ell\mid\tfrac{p^p-1}{e}$. In particular, $\ell^2\mid p^p-1$. Since $(\ell,p)=1$ and $p^p\equiv 1$ (mod $\ell$), Fermat's little theorem implies that $p\mid \ell-1$. Then 
$$
p^{\ell-1}= p^{p\cdot(\ell-1)/p} = (1+(p^p-1))^{(\ell-1)/p}\equiv 1\quad (\text{mod }\ell^2),
$$
which proves that $\ell$ is a Wieferich prime in base $p$. For readability, we break down the rest of the proof into the following two lemmata. 
 
 \begin{lemma}\label{A}
 Fix $m\geq0$. If  $f(X^{\ell^m})$ is irreducible over $\mathbb{F}_p$, then each polynomial in the tuple
$
(f_1(X^{\ell^m}),f_2(X^{\ell^m}),\dots,f_{p-1}(X^{\ell^m}))
$
is irreducible over $\mathbb{F}_p$.
\end{lemma}
 \begin{proof}
 Fix $a\in\mathbb{F}_p^\times$. Choose $b\in\mathbb{F}_p^\times$ such that $ba=-1$ in $\mathbb{F}_p$. Then
$$
b\cdot f_a(b^{-1}X^{\ell^m}) = b\left(b^{-p}X^{p\ell^m}-b^{-1}X^{\ell^m}+a\right) = X^{p\ell^m}-X^{\ell^m}+ba =f(X^{\ell^m}).
$$
%$f_a(X^k)$ is irreducible if and only if $c\cdot f_a(c^{-1}X^k)$ is irreducible for each $c\in\mathbb{F}_p^\times$. 
Since $(\ell,p-1)=1$, we have $b^{-1}=c^{\ell^m}$ for some $c\in\mathbb{F}_p^\times$. If $f_a(X^{\ell^m})$ is reducible, then there exist two nonconstant polynomials $g(X), h(X)\in\mathbb{F}_p[X]$ such that
$$
f(X^{\ell^m})=b\cdot f_a((cX)^{\ell^m})=b\cdot g(cX)h(cX). 
$$
Hence if $f(X^{\ell^m})$ is irreducible, then so is $f_a(X^{\ell^m})$.
\end{proof}

\begin{lemma}\label{B}
Fix $m\geq0$, and let $\beta:=\beta_{m,\ell}$ be a root of $f(X^{\ell^m})$. Then the multiplicative order $\mathrm{ord}(\beta)$ of $\beta$ in $\mathbb{F}_p(\beta)$ is $e\ell^m$. Moreover, $[\mathbb{F}_{p}(\beta):\mathbb{F}_p]=p\ell^m$ if and only if $\ell\nmid \tfrac{p^p-1}{e}.$
\end{lemma}
\begin{proof}
Let $d:=d_{m,\ell}$ be the smallest positive integer such that $\mathbb{F}_{p^d}=\mathbb{F}_p(\beta)$. Equivalently, $d=[\mathbb{F}_p(\beta):\mathbb{F}_p]$. Since $\mathrm{ord}(\beta)\mid |\mathbb{F}_{p^d}^\times|=p^d-1$, we note that $d$ is also the order of $p$ in $(\mathbb{Z}/\mathrm{ord}(\beta)\mathbb{Z})^\times$. To determine $\mathrm{ord}(\beta)$, we first observe that 
$$
\mathrm{ord}(\beta^{\ell^m})=\frac{\mathrm{ord}(\beta)}{(\ell^m,\mathrm{ord}(\beta))},
$$
which is a standard result for cyclic groups. Since $f(\beta^{\ell^m})=0$ and the order of $f(X)$ is $e$, we have
$$
e = \frac{\mathrm{ord}(\beta)}{(\ell^m,\mathrm{ord}(\beta))}.
$$
It follows that $\mathrm{ord}(\beta)\mid e\ell^m$, and the above equation is equivalent to $(\tfrac{e\ell^m}{\mathrm{ord}(\beta)},e)=1$. We conclude that $\mathrm{ord}(\beta)=e\ell^m$. 
Therefore $d_{m,\ell}$ is the order of $p$ in $(\mathbb{Z}/e\ell^m\mathbb{Z})^\times$. 

We claim that 
\begin{align}\label{key cong}
p^{p\ell^m} \equiv 1 +(p-1)Q\ell^m\ (\text{mod }e\ell^{m+1})
\end{align}
for each $m\geq0$. If $m=0$, this follows from the definition of $Q$. The claim then follows by induction, using the binomial theorem and that $\ell\mid\binom{\ell}{j}$ for each $0<j<\ell$. With this congruence relation in hand, we can now show by induction over $m\geq0$ that $d_{m,\ell}=p\ell^m$ if and only if $\ell\nmid \tfrac{p^p-1}{e}$.

If $m=0$, we have $p^p\equiv 1$ (mod $e$) and $p\not\equiv 1$ (mod $e$). Hence $d_{0,\ell}=p$. For $m>0$, we have $p^{p\ell^m} \equiv 1$ (mod $e\ell^m$), and hence $d_{m,\ell}\mid p\ell^m$. On the other hand, by definition of $d_{m,\ell}$, we have $p^{d_{m,\ell}}\equiv 1$ (mod $e\ell^{m-1}$), and hence $d_{m-1,\ell}\mid d_{m,\ell}$. The induction hypothesis $d_{m-1,\ell}=p\ell^{m-1}$ implies that $d_{m,\ell}$ is equal to either $p\ell^m$ or $p\ell^{m-1}$. To rule out the latter option, we deduce from (\ref{key cong}) that
$$
p^{p\ell^{m-1}}\equiv 1+(p-1)Q\ell^{m-1}\equiv 1\ (\text{mod }e\ell^m)
$$
if and only if $\ell\mid \tfrac{p^p-1}{e}$.
\end{proof}

If $\ell\nmid\tfrac{p^p-1}{e}$, then by Lemma \ref{B}, the minimal polynomial of $\beta$ over $\mathbb{F}_p$ has degree $p\ell^m$. Since this is also the degree of $f(X^{\ell^m})$, by the uniqueness of the minimal polynomial, we conclude that $f(X^{\ell^m})$ is irreducible. Lemma \ref{A} then finishes the proof of Theorem \ref{thm}.

\begin{acknowledgment}{Acknowledgments.}
The authors thank the reviewers for their many helpful comments and suggestions.
\end{acknowledgment}

\begin{biog}
\item[Claire Burrin] received her Ph.D.~from ETH Z\"urich in 2016. She spent four years as Hill Assistant Professor at Rutgers University before coming back to ETH Z\"urich as Senior Research Associate.
\begin{affil}
ETH Z\"urich, Z\"urich, Switzerland\\
claire.burrin@math.ethz.ch
\end{affil}

\item[Matthew Issac] is an undergraduate student in Mathematics at Rutgers University. Parts of this article came out of an undergraduate research project in 2019. 
\begin{affil}
Rutgers University, Piscataway NJ 08544\\
matthew.issac@rutgers.edu
\end{affil}
\end{biog}
\vfill\eject


\begin{thebibliography}{1}

\bibitem{BS2} Bary-Soroker, L. (2012). Hardy--Littlewood tuple conjecture over large finite fields. \textit{Int. Math. Res. Not. IMRN} 2014(2): 568--575. doi.org/10.1093/imrn/rns249

\bibitem{DK} Dorais, F., Klyve, D. (2011). A Wieferich prime search up to $6.7\times 10^{15}$. \emph{J. Integer Seq.}  14(9):1--14.

\bibitem{Eff} Effinger, G. (2008). Toward a complete twin primes theorem for polynomials. In: Mullen, G. L., Panario, D., Shparlinski, I. E., eds. {\it Finite Fields and Applications}. Contemp. Math., 461. Providence, RI: American Mathematical Society, pp. 103--110. doi.org/10.1090/conm/461/08986

\bibitem{Hall1} Hall, C. (2003). $L$-functions of twisted Legendre curves. Ph.D. dissertation. Princeton University, Princeton, NJ, USA.

\bibitem{Lang} Lang, S. (2002). \textit{Algebra}, rev. 3rd ed. Graduate Texts in Mathematics, 211. New York: Springer-Verlag.

\bibitem{LN} Lidl, R., Niederreiter, H. (1997). \textit{Finite Fields}, 2nd ed. Encyclopedia of Mathematics and Its Applications, 20. Cambridge: Cambridge Univ.~Press.

\bibitem{Ma} Maynard, J. (2015). Small gaps between primes.  \textit{Ann. of Math. (2).} 181(1):  383--413. doi.org/10.4007/annals.2015.181.1.7

\bibitem{Mi} Mih\u{a}ilescu, P. (2004). Primary cyclotomic units and a proof of Catalan's conjecture. \textit{J. Reine Angew. Math.}  572: 167--195. doi.org/10.1515/crll.2004.048

\bibitem{Mo} Montgomery, P. L., Nahm, S., Wagstaff, S. (2010). The period of the Bell numbers modulo a prime. {\it Math. Comp.} 79(1): 1793--1800. doi.org/10.1090/S0025-5718-10-02340-9

\bibitem{Pol}  Pollack, P. (2008). Prime polynomials over finite fields. Ph.D. dissertation. Dartmouth College, Hanover, NH, USA.

%\bibitem{PrimeGrid} PrimeGrid PRPNet (2015). "Wieferich Prime Search." \url{prpnet.primegrid.com:13000}.

\bibitem{Rota} Rota, G.-C. (1964). The number of partitions of a set. {\it Amer. Math. Monthly}. 71(5): 498--504. doi.org/10.2307/2312585

\bibitem{SawinShusterman} Sawin, W., Shusterman, M. (2019). On the Chowla and twin primes conjectures over $\mathbb{F}_q[T]$. arxiv.org/abs/1808.04001

\bibitem{Sil} Silverman, J. (1988). Wieferich's criterion and the $abc$-conjecture. {\it J. Number Theory}. 30(2): 226--237. doi.org/10.1016/0022-314X(88)90019-4

\bibitem{Wi} Wieferich, A. (1909). Zum letzten Fermatschen Theorem. {\it J. Reine Angew. Math.} 136: 293--302. doi.org/10.1515/crll.1909.136.293

\bibitem{Zh}  Zhang, Y. (2014). Bounded gaps between primes. \textit{Ann. of Math. (2).} 179(3): 1121--1174. doi.org/10.4007/annals.2014.179.3.7

\end{thebibliography}
\end{document}